\documentclass[12pt,reqno]{amsart}

  \usepackage{amsmath}
  \usepackage{amsthm}
  \usepackage{amsfonts}
  \usepackage{amssymb}
  \usepackage{amscd}
  \usepackage{amsbsy}

\usepackage[margin=1.3in]{geometry}
\usepackage{xcolor}
\usepackage{enumerate}
\usepackage[sort, numbers]{natbib}
 
\usepackage[
	colorlinks,
	linkcolor={black},
	citecolor={black},
	urlcolor={black}
]{hyperref}

  \newtheorem{theorem}{Theorem}[section]
  \newtheorem{lemma}[theorem]{Lemma}
  \newtheorem{proposition}[theorem]{Proposition}
  \newtheorem{claim}[theorem]{Claim}
  \newtheorem{corollary}[theorem]{Corollary}
  
  \theoremstyle{definition}

  \newtheorem*{remark}{Remark}

  \newtheorem{step}{Step}[theorem]

\newenvironment{claimproof}[1][Proof of Claim]{\noindent \underline{#1.} }{\hfill$\diamondsuit$}

  \numberwithin{equation}{section}

\newcommand{\bbC}{{\mathbb{C}}}
\newcommand{\bbD}{{\mathbb{D}}}

\newcommand{\bbR}{{\mathbb{R}}}
\newcommand{\bbT}{{\mathbb{T}}}
\newcommand{\bbZ}{{\mathbb{Z}}}

\newcommand{\frA}{{\mathfrak{A}}}
\newcommand{\frF}{{\mathfrak{F}}}
 
\newcommand{\Aut}{{\mathrm{Aut}}} 
 
\DeclareMathOperator{\supp}{{supp}}
\DeclareMathOperator{\tr}{{Tr}}
 
 \newcommand{\set}[1]{{\left\{#1\right\}}}

  \allowdisplaybreaks
\title{The Schwartzman Group of an Affine Transformation}

\author[D.\ Damanik]{David Damanik}
\address{Department of Mathematics, Rice University, Houston, TX~77005, USA}
\email{damanik@rice.edu}

\author[\'{I}.\ Emilsd\'{o}ttir]{\'{I}ris Emilsd\'{o}ttir}
\address{Department of Mathematics, Rice University, Houston, TX~77005, USA}
\email{irist@rice.edu}

\author[J.\ Fillman]{Jake Fillman}
\address{Department of Mathematics, Texas State University, San Marcos, TX 78666, USA}
\email{fillman@txstate.edu}

\begin{document}

\begin{abstract}
We compute the Schwartzman group associated with an ergodic affine automorphism of a compact connected abelian group given by the composition of an automorphism of the group and a translation by an element in the path component of the identity. We show that the Schwartzman group can be characterized by evaluating the invariant characters of the automorphism at the group element by which one translates. As a byproduct, we show that the set of labels associated with the doubling map on the dyadic solenoid is trivial, which in turn allows us to show that any ergodic family of Jacobi matrices defined over the doubling map has connected almost-sure essential spectrum.
\end{abstract}

\maketitle



\hypersetup{
	linkcolor={black!30!blue},
	citecolor={red},
	urlcolor={black!30!blue}
}

\section{Introduction}
We are interested in the study of ergodic Schr\"odinger operators of the form
\begin{equation} \label{eq:Homegadef}
[H_\omega u] = u(n-1) + u(n+1) + f(T^n\omega) u(n), \quad n \in \bbZ, 
\end{equation}
where $\omega \in \Omega$, $\Omega$ denotes a compact metric space, $T:\Omega \to \Omega$ is a homeomorphism, and $f \in C(\Omega,\bbR)$. In this case, we call $(\Omega,T)$ a topological dynamical system. If $\mu$ denotes a $T$-ergodic Borel probability measure on $\Omega$, then there is a set $\Sigma = \Sigma_\mu$, called the \emph{almost-sure spectrum} of the family, such that $\sigma(H_\omega) = \Sigma$ for $\mu$-a.e.\ $\omega \in \Omega$; see, for example, \cite{Damanik2017ESOSurvey, ESO1} for background.

Naturally, since the operators $H_\omega$ are bounded and self-adjoint, $\Sigma$ is a compact subset of $\bbR$, so its complement is the union of at most countably many disjoint open intervals, called the \emph{gaps} of the spectrum. The spectrum itself can have many different topological structures: it may be connected, it may be totally disconnected, it may consist of a finite union of nondegenerate closed intervals, and so on. This naturally leads to interest in the structure of the complement $\bbR \setminus \Sigma$, that is, in the structure of the gaps. The \emph{gap-labelling theorem} gives an invaluable tool in the study of the gaps. 

Associated with the family $\{H_\omega\}_{\omega \in \Omega}$, there is a function $k$ called the integrated density of states (IDS), which computes the average proportion of eigenvalues of cutoffs of $H_\omega$ that lie below a given threshold; see \eqref{e.IDSv2} below for the definition of the IDS. The crucial feature of the IDS is that the spectrum is precisely the set of growth points of the function $k$. In particular, the IDS is constant on each connected component of $\bbR \setminus \Sigma$, and it assumes different constant values on different connected components. The value assumed by $k$ on a gap is called the \emph{label} of the gap. The gap-labelling theorem then asserts that there is a countable subgroup of $\bbR$ that only depends on $(\Omega,T,\mu)$ such that all labels must belong to this group.

There are some different versions of gap-labelling. One version, due to Bellissard and coworkers \cite{Bel1986}, identifies a set of labels with a normalized trace on a suitable $C^*$ algebra. The version due to Johnson identifies a set of labels with the range of a particular homomorphism \cite{Johnson1986JDE}. This group will be denoted by $\frA(\Omega,T,\mu)$, called the Schwartzman group \cite{Schwarzmann1957Annals}, and defined precisely in Subsection~\ref{sec:schwartzman}. Given the relationship between $\frA(\Omega,T,\mu)$ and operators defined by the dynamical system $(\Omega,T)$ via \eqref{eq:Homegadef}, it is naturally of interest to compute $\frA(\Omega,T,\mu)$ for as many dynamical systems as possible. This was carried out in many standard examples in \cite{DFGap}.

Our main result computes Schwartzman groups associated with affine transformations of suitable groups. Recall that a topological group is a Hausdorff topological space  that is also a group for which the group operations (multiplication and inversion) are continuous. If $\Omega$ is a topological group, we write $\Aut(\Omega)$ for the set of continuous group automorphisms $A: \Omega \to \Omega$. If $A \in \Aut(\Omega)$ and $b \in \Omega$, we write $T_{A,b}:\omega \mapsto A\omega +b$ for the corresponding affine automorphism. We denote by $\widehat{\Omega}$ the group of continuous homomorphisms from $\Omega$ to $\bbT := \bbR/\bbZ$, which is called the Pontryagin dual group of $\Omega$. Naturally, any $A \in \Aut(\Omega)$ induces a dual map $\widehat{A}:\widehat{\Omega} \to \widehat{\Omega}$ via $\widehat{A}\chi = \chi\circ A$ for $\chi \in \widehat{\Omega}$. For additional background about topological groups and their duals, see textbook treatments in \cite{HewittRoss1979, Morris1977:TopGrps, Rudin1990Fourier}.

Let us now state our main result. Here and throughout the paper, $\pi$ denotes the canonical projection $\bbR \to \bbR/\bbZ$ sending $x \in \bbR$ to its equivalence class modulo $\bbZ$ in $\bbR/\bbZ$.

\begin{theorem}\label{t:main}
If $\Omega$ is a compact connected abelian group, ${A \in \Aut(\Omega)}$, $b \in \Omega$ belongs to the path component of the identity of $\Omega$, and $\mu$ is a $T_{A,b}$-ergodic probability measure on $\Omega$, then the Schwartzman group of $(\Omega,T_{A,b}, \mu)$ is given by
\begin{align} \label{eq:AforAffine}
\frA(\Omega,T_{A,b},\mu) 
& = \set{ \beta \in \bbR :  \beta \in \pi^{-1}(\chi b) \text{ for some } \chi \in \ker(\widehat{A}-I) } \\
\nonumber
& = \bigcup_{\chi \in \ker(\widehat A - I)} \pi^{-1}(\chi b).
\end{align}
\end{theorem}

\begin{remark}
Let us make a few comments about the theorem.
\begin{enumerate}
\item[{\rm (a)}] Since the torus $\bbT^d$ is a compact connected abelian group (and every $b \in \bbT^d$ of course belongs to the path component of the identity), Theorem~\ref{t:main} is a generalization of \cite[Theorem~8.1]{DFGap}.
\smallskip

\item[{\rm (b)}] There are compact connected abelian groups that are not tori and are not path-connected. In particular, the assumption that $b$ lies in the path component of $0$ is not automatically satisfied, even though $\Omega$ is assumed to be connected.
\smallskip

\item[{\rm (c)}] The assumption of connectedness in Theorem~\ref{t:main} cannot be dropped entirely. We give a simple detailed example in the Appendix; see Corollary~\ref{coro:disconExample}. 
\smallskip

\item[{\rm (d)}] The proof we give requires the assumption that $b$ lies in the path component of the identity. We regard it as an interesting problem to determine whether the result needs this assumption.
\smallskip

\item[{\rm (e)}] If $A = I$, then $\ker(\widehat{A} - I) = \widehat{\Omega}$, the whole dual group of $\Omega$. In this case, \eqref{eq:AforAffine} equates the Schwartzman group of the translation $\omega \mapsto \omega+b$ with the frequency module of $(\Omega,b)$. Thus, one recovers a special case of the classical gap-labelling theorem for almost-periodic dynamical systems. In particular, in the case $A = I$, the result holds for all compact groups $\Omega$ and all $b \in \Omega$ without any connectedness assumptions; compare \cite{DFGap, DelSou1983CMP, JohnMos1982CMP}. 
\end{enumerate}
\end{remark}
The proof of Theorem~\ref{t:main} follows the general contours of the proof of \cite[Theorem~8.1]{DFGap}. However, it yields some surprisingly powerful conclusions. For instance, Theorem~\ref{t:main} allows us to answer several questions related to operators defined by the doubling map, and indeed the answers that we get would have been surprising in a vacuum. More specifically, \cite{DF22Doubling} shows that the almost-sure essential spectra of Schr\"odinger operators defined by the doubling map are connected, and did so by proving triviality of the range of the Schwartzman group restricted to maps on the solenoid factoring through the doubling map. However, we will show in Corollary~\ref{coro:2adicSchwartzGroup} that the result above implies triviality of the range of the Schwartzman homomorphism for arbitrary maps on the standard solenoid. At the time \cite{DF22Doubling} was written, the authors suspected that this would not be the case. Similarly,  \cite{DFZ} studied Jacobi matrices defined by the doubling map and proved connectedness of the almost-sure essential spectra for such operators under the assumption of nonvanishing off-diagonals. We will use Theorem~\ref{t:main} to show that the results of \cite{DFZ} indeed hold for arbitrary off-diagonals and hence hold in maximal generality; see Corollary~\ref{coro:JacobiDoubling}.

We will give some background in Section~\ref{sec:background}, prove Theorem~\ref{t:main} in Section~\ref{sec:computingA}, and discuss some interesting applications in Sections~\ref{sec:solenoid} and \ref{sec:Jacobi}.

\section*{Acknowledgements}
D.\ D.\ was supported in part by NSF grant DMS--2054752. J.\ F.\ was supported in part by Simons Foundation Collaboration Grant \#711663 and NSF grant DMS--2213196. The authors gratefully acknowledge support from the American Institute of Mathematics, ICERM, and the Simons Center for Geometry and Physics, at which portions of this work were done.

\section{Background} \label{sec:background}

\subsection{Topological Dynamics and the Schwartzman Group} \label{sec:schwartzman}
We define a \emph{topological dynamical system} to consist of a pair $(\Omega,T)$ in which $\Omega$ is a compact metric space and $T:\Omega \to \Omega$ is a homeomorphism; in particular, we consider invertible dynamics. A Borel probability measure on $\Omega$ is called $T$-\emph{invariant} if $\mu(T^{-1}B) = \mu(B)$ for all Borel sets $B$ and $T$-\emph{ergodic} if it is $T$-invariant and any $T$-invariant measurable function is a.e.\ constant.

Given an ergodic topological dynamical system $(\Omega,T,\mu)$, its \emph{suspension} is given by $X = \Omega \times \bbR /  \sim$, where $(T^n\omega,t) \sim (\omega,t+n)$ for $(\omega,t) \in \Omega \times \bbR$ and $n \in \bbZ$. This can be made into a continuous-time dynamical system via
\[\tau^s([\omega,t]) = [\omega,t+s],\]
where $[\omega,t]$ denotes the equivalence class of $(\omega,t)$ in $X$. Likewise, the ergodic measure $\mu$ induces a $\tau$-ergodic measure, $\nu$, on $X$ via
\begin{equation}
\int_X g \, d\nu = \int_\Omega \int_0^1 g([\omega,t]) \, dt \, d\mu(\omega).
\end{equation}

Given $g \in C(X,\bbT)$, one can lift the function $g_x: t \mapsto g(\tau^tx)$ to a map $\widetilde{g}_x:\bbR \to \bbR$. The limit
\[  \lim_{t\to\infty} \frac{\widetilde{g}_x(t)}{t} \]
exists for $\nu$-a.e.\ $x$, is $\nu$-a.e.\ constant, and only depends on the homotopy class of $g$. Denoting by $C^\sharp(X,\bbT)$ the set of homotopy classes of maps $X \to \bbT$, the induced map $\frF_\nu:C^\sharp (X,\bbT) \to \bbR$ is called the \emph{Schwartzman homomorphism} and its range is known as the \emph{Schwartzman group}, denoted
\begin{equation} \label{eq:schwartzmanDef}
\frA(\Omega,T,\mu) = \frF_\nu(C^\sharp(X,\bbT)).
\end{equation}
For later use (in the proof of Lemma~\ref{Lemma1}), it is helpful to note that two maps $X\to \bbT$ are homotopic if and only if their difference lifts to a map $X \to \bbR$. Thus,  $C^\sharp(X,\bbT)$ is equivalent to $C(X,\bbT)/H(X,\bbT)$ where $H(X,\bbT)$ denotes the subgroup of $C(X,\bbT)$ consisting of maps of the form $x\mapsto \pi(f(x))$ where $f:X \to \bbR$ is continuous and $\pi:\bbR \to \bbR / \bbZ$ is the standard projection; compare \cite[Proposition~4.10]{DFGap}.

\subsection{Gap Labels for Jacobi Matrices}
We also will discuss ergodic Jacobi matrices and the topological structure of their almost-sure spectrum.  To introduce these objects, let $(\Omega,T)$ be a topological dynamical system, that is, $\Omega$ is a compact metric space and $T : \Omega \to \Omega$ is a homeomorphism. Given $q \in C(\Omega,\bbR)$ and $p \in C(\Omega,\bbC)$, we consider the family of Jacobi matrices $\{J_\omega\}_{\omega \in \Omega}$, acting in $\ell^2(\bbZ)$, given by 
\begin{equation}\label{e.jacmatv2}
[J_\omega\psi](n) = \overline{p(T^{n-1}\omega)}\psi(n-1)+q(T^n\omega)\psi(n) + p(T^n\omega)\psi(n+1).
\end{equation}
The discrete Schr\"odinger operators in \eqref{eq:Homegadef} are a special case of ergodic Jacobi matrices with $q = f$ and $p \equiv 1$.

Fix a $T$-ergodic Borel probability measure $\mu$ on $\Omega$. We assume 
\begin{equation} \label{eq:fullsuppv2}
\supp\mu = \Omega,
\end{equation} 
where $\supp \mu$ denotes the topological support of $\mu$, that is, the smallest closed set having full $\mu$-measure. Let us mention that assumption \eqref{eq:fullsuppv2} is non-restrictive, since one can always replace the dynamical system $(\Omega,T)$ by $(\supp\mu,T|_{\supp\mu})$. Having chosen the measure $\mu$, we call $\{J_\omega\}$ an \emph{ergodic family of Jacobi matrices}. The \emph{density of states measure} (DOSM) is given by
\begin{equation}\label{e.DOSMv2}
\int f \, d\kappa 
= \lim_{N \to \infty} \int f\, d\kappa_{\omega,N}
:= \lim_{N\to\infty} \frac{1}{N} \tr(f(J_\omega\chi_{_{[0,N)}})), \quad \mu\text{-a.e.\ } \omega \in \Omega,
\end{equation}
and the \emph{integrated density of states} (IDS) is then given by
\begin{equation}\label{e.IDSv2}
k(E) = \int \! \chi_{_{(-\infty,E]}} \, d\kappa.
\end{equation}

We have the following gap-labelling result for ergodic families of Jacobi matrices. 

\begin{theorem}[{\cite[Theorem~1.1]{DFZ}}] \label{t:DFZmain}
Suppose $(\Omega,T)$ is an invertible topological dynamical system and $\mu$ is a fully supported $T$-ergodic Borel probability measure on $\Omega$. Given $p \in C(\Omega,\mathbb{C})$ and $q \in C(\Omega,\mathbb{R})$, let $\{J_\omega\}$ denote the associated ergodic family of Jacobi matrices, $\Sigma$ the almost-sure spectrum of this family, and $k$ its IDS. For all $E \in \bbR \setminus \Sigma$, we have
\begin{equation}
k(E) \in \frA(\Omega,T,\mu).
\end{equation}
\end{theorem}

\subsection{The Dyadic Solenoid} \label{subsec:solenoid}
In order to demonstrate that Theorem~\ref{t:main} contains new content, we show that it can be applied to solenoids. To that end, let us introduce three realizations of the dyadic solenoid, which we will denote by $\mathcal{S}_{1},\mathcal{S}_{2} \text{ and }\mathcal{S}_{3}$.
The first representation can be defined by 
\begin{equation} \mathcal{S}_{1} := (\mathbb{R} \times \mathbb{Z}_2)/A
\end{equation}
 where $\bbZ_2$ denotes the $2$-adic integers, and $A$ denotes the discrete, hence closed, subgroup $A= \{(a,-a) \text{ : } a\in \mathbb{Z}\} \subseteq \bbR \times \bbZ_2$. The doubling map on $\mathcal{S}_{1}$ is given by
 \begin{equation}
 T_1 [r,s] = [2r,2s].
 \end{equation} 
In the appendix of Chapter~1 in \cite{Robert2000padic}, it is shown that this is equivalent to the inverse limit 
\begin{equation}
\mathcal{S}_{2}:=\varprojlim  \mathbb{R}/2^n \mathbb{Z}
= \set{(x_n)_{n \geq 0} \in \prod_{n \geq 0} \bbR/2^n \bbZ : \sigma_n x_{n+1} = x_n \text{ for all } n}
\end{equation} 
of the projective system $(\mathbb{R} / 2^n \mathbb{Z}, \sigma_n)$, where $\sigma_n$ denotes the canonical projection $\sigma_n: x \text{ mod } 2^{n+1}\bbZ \mapsto x \text{ mod } 2^n\bbZ \text{, }n\geq 0$. The doubling map on $\mathcal{S}_2$, is defined by 
\begin{equation}
(T_2 x)_n = 2x_n.
\end{equation}

Finally, in Chapter~1 of \cite{BrinStuck2015}, the solenoid is realized as the attractor of an iterated function system on the solid torus. More precisely, let $\mathcal{I} = \bbT \times \overline{\bbD}$ denote the solid torus (where $\overline{\bbD} = \{(x,y) \in \bbR^2 : x^2+y^2 \leq 1\}$). Choosing some $\lambda \in (0,1/2)$, the transformation $F:\mathcal{I}\rightarrow \mathcal{I}$ is given by 
\begin{equation}
F(\omega,x,y)=  \left(2\omega,\lambda x+ \frac{1}{2}\cos(2\pi \omega),\lambda y + \frac{1}{2} \sin(2\pi \omega) \right).
\end{equation} The solenoid is then given by the attractor of this system, that is,
\begin{equation}
\mathcal{S}_{3}:=\bigcap_{n=0}^{\infty}F^n(\mathcal{I}).
\end{equation} 
The restriction of $F$ to $\mathcal{S}_3$ represents the doubling map, and is denoted by $T_3$.

\begin{proposition} \label{prop:allSolenoidsEquiv}
All three realizations of the doubling map on the dyadic solenoid are topologically conjugate to one another.
\end{proposition}

\begin{proof}
This statement is well known in the literature. For the reader's convenience, we sketch the main steps and provide references. It is mentioned in \cite{Robert2000padic} that the definitions of $\mathcal{S}_{1}\text{ and }\mathcal{S}_{2}$ are equivalent and in \cite{BrinStuck2015} the equivalence between $\mathcal{S}_{3}$ and $\mathcal{S}_{2}$ is discussed. Let us describe this in more detail. 

Letting $\pi_n:\bbT \times \overline{\bbD} \to \bbR/2^n\bbZ$ be the map $\pi_n(\theta,r,s)= 2^n\theta \ \mathrm{mod} \ 2^n\bbZ$, we define $h:\mathcal{S}_3 \to  \mathcal{S}_{2}\subseteq \Pi_{n \geq 0 } \bbR / 2^n\bbZ$ by 
\begin{equation} (h(\omega,x,y))_n=\pi_n (T_3^{-n}(\omega,x,y)),\end{equation}
which can be done because $T_3$ is invertible on $\mathcal{S}_{3}$.      
         One can then check that this $h$ is a homeomorphism that maps $\mathcal{S}_3$ to $\mathcal{S}_2$ and satisfies $T_2 \circ h = h\circ T_3$. See \cite[Section~1.9]{BrinStuck2015} for additional discussion.
         
         Next we discuss a conjugacy $\mathcal{S}_1 \to \mathcal{S}_2$. Writing $\bbZ_+ = \{0,1,2,\ldots\}$ and writing a typical element $z\in \bbZ_2$ as $\sum_{j \in \bbZ_+} z_j 2^j$, define $g:\bbR \times \bbZ_2 \to \mathcal{S}_2$ by 
\begin{equation} 
(g(r,z))_n =   r+\sum_{j=0}^{n-1}z_{j}2^j , \quad  n \in \bbZ_+.
\end{equation}
The reader can verify that this is a well-defined homomorphism with $\ker(g)=A$, which descends to a map $\bar g:\mathcal{S}_1 \to \mathcal{S}_2$ satisfying $T_2 \circ \bar  g = \bar  g \circ T_1$. See the appendix to \cite[Chapter~1]{Robert2000padic} for further details.
\end{proof}

Since all three realizations of the doubling map on the solenoid are topologically equivalent, they all have the same Schwartzman group, which can be seen from the following proposition.

If $(\Omega_1,T_1,\mu_1)$ and $(\Omega_2,T_2,\mu_2)$ are ergodic topological dynamical systems, we say that they are \emph{equivalent} if $(\Omega_1,T_1)$ and $(\Omega_2,T_2)$ are topologically conjugate via a homeomorphism $h:\Omega_1 \to \Omega_2$ such that $h_* \mu_1 = \mu_2$.

\begin{proposition} \label{prop:schwartzConj}
If $(\Omega_1,T_1,\mu_1)$ and $(\Omega_2,T_2,\mu_2)$ are equivalent ergodic topological dynamical systems, then
\begin{equation}
\frA(\Omega_1,T_1,\mu_1) =\frA(\Omega_2,T_2,\mu_2).
\end{equation}
\end{proposition}

\begin{proof}
Let $h:\Omega_1 \to \Omega_2$ be as in the definition of equivalence. Letting $(X_j,\tau_j,\nu_j)$ denote the corresponding suspensions, $h$ induces a map $\overline{h}: X_1 \to X_2$ via $[\omega_1,t] \mapsto [h\omega_1,t]$.  Since 
\[
\overline{h}([T_1^n\omega_1,t-n]) 
= [hT_1^n\omega_1,t-n] 
= [T_2^n h\omega_1,t-n] 
= [h\omega_1,t] 
= \overline{h}([\omega_1,t]),\]
we see that $\overline{h}$ is a well-defined homeomorphism. A similar calculation shows $\overline{h} \circ \tau_1^s = \tau_2^s \circ \overline{h}$ for all $s$. Finally, one has $\overline{h}_* \nu_1 = \nu_2$ on account of
\begin{align*} \int f \, d(\overline{h}_*\nu_1)
& = \int f\circ \overline{h} \, d \nu_1 
= \int_0^1 \int_{\Omega_1} f([h\omega_1,t]) \, d\mu_1(\omega_1)\, dt \\
& = \int_0^1 \int_{\Omega_2} f([\omega_2,t]) \, d\mu_2(\omega_2)\, dt = \int f \, d\nu_2,
\end{align*} 
where we used $h_*\mu_1 = \mu_2$ in the penultimate step. Thus, $\overline{h}$ gives a topological conjugacy from $(X_1,\tau_1,\nu_1)$ to $(X_2,\tau_2,\nu_2)$.
This in turn gives a map $P:C(X_2,\bbT) \to C(X_1,\bbT)$ by defining
\begin{equation}
Pf = f\circ \overline{h}, \quad f \in C(X_2,\bbT).
\end{equation}

Since $\overline{h}$ is a homeomorphism, $P$ is a bijection that maps homotopic functions to homotopic functions and thus induces a bijection $P^\sharp:C^\sharp(X_2,\bbT) \to C^\sharp(X_1,\bbT)$.
\begin{claim} 
\begin{equation} \label{eq:schwarzConj:Fnu1nu2}
\mathfrak{F}_{\nu_2}([\phi]) = \mathfrak{F}_{\nu_1}(P^\sharp[\phi]) \textup{ for every } \phi \in C(X_2,\bbT).
\end{equation}
\end{claim}

\begin{claimproof}
For $x_2 \in X_2$, recall $\phi_{x_2}(t) = \phi(\tau_2^t x_2)$, and that $\widetilde\phi_{x_2}:\bbR \to \bbR$ denotes a lift. Note that 
\[ (P\phi)_{x_1}(t) = (P\phi)(\tau_1^t x_1) = \phi(\overline{h}\tau_1^tx_1) = \phi(\tau_2^t \overline{h}x_1) = \phi_{\overline{h}x_1}(t).\]
In particular, $\widetilde{\phi}_{\overline{h}x_1}$ is a lift of $(P\phi)_{x_1}$, so we take $\widetilde{(P\phi)}_{x_1} = \widetilde{\phi}_{\overline{h} x_1}$.
For $j=1,2$, choose sets $X_j' \subseteq X_j$ of full $\nu_j$-measure such that one has
\[ \frF_{\nu_1}([P\phi]) = \lim_{t\to\infty} \frac{\widetilde{(P\phi)}_{x_1}(t)}{t}, \quad
\frF_{\nu_2}([\phi]) = \lim_{t\to\infty} \frac{\widetilde\phi_{x_2}(t)}{t}\]
for all $x_j \in X_j'$. On account of $\overline{h}_*\nu_1=\nu_2$, the set $X_2^\star = X_2' \cap \overline{h}[X_1']$ has full $\nu_2$-measure. Choosing $x_2 \in X_2^\star$, one has $\overline{h}^{-1} x_2 \in X_1'$ and thus
\begin{align*}
\frF_{\nu_2}([\phi])
 = \lim_{t \to \infty} \frac{\widetilde{\phi}_{x_2} (t)}{t} 
 = \lim_{t \to \infty} \frac{\widetilde{(P\phi)}_{\overline{h}^{-1}x_2} (t)}{t} 
 = \frF_{\nu_1}([P\phi]),
\end{align*}
which proves the claim.
\end{claimproof}

With the claim proved, we see that the ranges of $\frF_{\nu_1}$ and $\frF_{\nu_2}$ are identical, as desired.
\end{proof}

\section{Computing the Schwartzman Group} \label{sec:computingA}
As in the proof of \cite[Theorem~8.1]{DFGap} we structure the argument via two lemmas. First,  we characterize the homotopy classes from the suspension of $(\Omega,T_{A,b},\mu)$ to the circle $\bbT$. We then use this to compute the relevant Schwartzman group.
To that end, let us fix a compact connected abelian group $\Omega$, and write $X_b:=\Omega \times[0,1]/ ((\omega,1)\sim (T_{A,b}\omega, 0)) $ for the suspension of $(\Omega,T_{A,b})$.

\begin{lemma}
\label{Lemma1}
Suppose $\Omega$ is a compact connected abelian group, $A \in \Aut(\Omega)$, and $b \in \Omega$. For $\chi \in {K}:= \ker(I-\widehat{A})$ and $\beta \in \pi^{-1}(\chi b)$ define $g_{\chi,\beta}: X_b\rightarrow \bbT$ by 
\begin{equation}\label{eq:gknbDef}
 g_{\chi,\beta}([\omega,t]) = \chi \omega+\beta t \ \mathrm{mod} \ \bbZ.\end{equation}
\begin{enumerate}[{\rm(a)}]
\item For each $\chi \in {K}$ and $\beta \in \pi^{-1}(\chi b)$, $g_{\chi,\beta}$ is a well-defined continuous map.
\item Every $g\in C(X_0,\bbT)$ is homotopic to $g_{\chi,\beta}$ for some $\chi \in K$ and $\beta \in \bbZ$.\footnote{Here, $\bbZ$ arises because $b=0$ and $\pi^{-1}(\{\chi 0\}) = \pi^{-1}(\{0\}) = \bbZ$.}
\item If $b$ is in the path component of the identity of $\Omega$, then $X_b$ is homeomorphic to $X_0$.
\item If $b$ is in the path component of the identity of $\Omega$, one has 
\[C^{\sharp} (X_b,\bbT)=\{[g_{\chi,\beta}] : \chi\in {K},\  \beta \in \pi^{-1}(\chi b)]\}.\]
\end{enumerate}
\end{lemma}
\begin{proof}
(a) Given $\chi \in K$ and $\beta \in \pi^{-1}(\chi b)$, $\chi \circ A = \chi$ and $\chi b= \beta \ \mathrm{mod} \ \bbZ$ imply
\[g_{\chi,\beta}([\omega,1]) = \chi \omega + \beta
= \chi(A\omega+b) = g_{\chi,\beta}([A\omega+b,0]),\]
which suffices to show that $g_{\chi,\beta}$ is well-defined and continuous.
\bigskip

(b) Let $g \in C(X_0,\bbT)$ be given. We show that $g$ is homotopic to some $g_{\chi,\beta}$ in a sequence of steps. 
\begin{step} For each $t$, let $\Omega_t = \{ [\omega,t] : \omega\in \Omega  \}$ represent the corresponding fiber in $X_0$. Since $\Omega$ is compact and connected, every continuous map from $\Omega$ to $\bbT$ is  homotopic to exactly one element in $\widehat{\Omega}$ \cite{Scheffer1972} (see also \cite{Plunkett1953MMJ} for earlier work under stronger assumptions). Consequently, for each $t \in \bbT$, the map $g^{(t)}:\Omega \to \bbT$ given by 
\begin{equation}
g^{(t)}(\omega) = g([\omega,t])
\end{equation}
 is homotopic to a unique $\chi_t \in \widehat\Omega$. Since $g^{(t)}$ is homotopic to $g^{(s)}$ for all $t$ and $s$, there is a single $\chi \in \widehat\Omega$ with $\chi = chi_t$ for all $t$. \end{step} 
 
 \begin{step}
 Since $g^{(0)}$ is homotopic to $\chi$ and to $g^{(1)} = g^{(0)} \circ A$, $\chi$ is homotopic to $\chi \circ A$, which by \cite{Scheffer1972} implies $\chi \circ A = \chi$, that is, $\chi \in K$. \end{step} 

\begin{step} Consider the circle $S_0 = \{[0,t]: t \in \bbR\}$. By a standard fact from topology, there exists $\beta \in \bbZ$ such that $[0,t] \mapsto g([0,t])$ is homotopic to the map $[0,t] \mapsto \beta t$. \end{step} 
 
\begin{step} Let us take the $\chi$ from Step~1 and $\beta$ from Step~3. By construction, there exists $g_\star \in C(X_0,\bbT)$ homotopic to $g$ such that $g_\star[\omega,0]= \chi \omega$ and $g_\star[0,t]=\beta t$. Since $g$ is homotopic to $g_\star$, note that $g_\star^{(t)}: \omega \mapsto g_\star([\omega,t])$ is homotopic to $\chi$ for all $t$. \end{step} 

\begin{step} Define $h := g_\star - g_{\chi,\beta}$. From the definitions of $g_\star$ and $g_{\chi,\beta}$ we know that $h$ vanishes on the set $\Omega_0 \cup S_0$ and $h^{(t)}: \omega \mapsto h([\omega,t])$ is nullhomotopic for every $t$.

We will make use of the following fact: a map $f$ from a compact metric space $Y$ to $\bbT$ is nullhomotopic if and only if there exists a continuous function $\widetilde f:Y \to \bbR$ such that $\pi \circ \widetilde f = f$ where $\pi:\bbR \to \bbT$ denotes the canonical projection (in fact we will apply this principle to both $Y =\Omega$ and $Y=X$ below); see, e.g., \cite[Prop.\ 4.10]{DFGap} for a proof. For each $t \in [0,1]$, there is a continuous $\widetilde{h}^{(t)}:\Omega \to \bbR$ such that 
\begin{equation}
\pi \circ \widetilde h^{(t)} = h^{(t)} \text{ and }\widetilde{h}^{(t)}(0) = 0
\end{equation}
by \cite[Prop.\ 4.10]{DFGap}. Moreover, each of these lifts is unique: any two lifts of $h^{(t)}$ must differ by a locally constant function, which is then necessarily constant by connectedness of $\Omega$. Then, defining $\widetilde{h}([\omega,t]) = \widetilde{h}^{(t)}(\omega)$, we see that $\widetilde h$ satisfies $\pi \circ \widetilde h = h$. Moreover, we claim that $\widetilde h$ is continuous. To see this, consider for each $\omega \in \Omega$ 
$$
f_\omega:[0,1] \to \bbT ,\quad s \mapsto h([\omega,s]).
$$ 
For each $\omega$, there is a unique lift $\widetilde f_\omega:[0,1] \to \bbR$ with $\pi \circ \widetilde f_\omega = f_\omega$ and $\widetilde f_\omega(0) =0$.  Since $h([0,t])=0$ for every $t$, we have $f_0(s)=0$ for all $s$. Putting everything together, for each $s \in [0,1]$, $\omega \mapsto \widetilde f_\omega(s)$ is continuous and satisfies $\pi(\widetilde f_\omega(s)) = h([\omega,s])$ and thus by uniqueness, one has $\widetilde f_\omega(s) = \widetilde h^{(s)}(\omega)$ for all $s$ and $\omega$. Moreover, by uniform continuity of $h$, $\omega\mapsto \widetilde f_\omega \in C([0,1])$ is continuous if $C([0,1])$ is given the uniform topology. The continuity of $\widetilde h$ follows.  At last, this implies that  $h$ is nullhomotopic. \end{step} 
Since $h = g_\star - g_{\chi,\beta}$ is nullhomotopic and $g_\star$ is homotopic to $g$, it follows that $g$ is homotopic to $g_{\chi,\beta}$ and we are done.
\bigskip

(c) Assume $b$ is in the path component of the identity, and choose $\gamma:[0,1]\to \Omega$ continuous with $\gamma(0)=0$ and $\gamma(1)=b$. Note that $\gamma' := A^{-1} \circ \gamma$ gives a path from $0$ to $A^{-1}b$. One can then define the desired homeomorphism $\varphi:X_b \to X_0$ via
\begin{equation}
\varphi([\omega,t]) = [\omega+\gamma'(t),t].
\end{equation}
By construction $(\omega,t) \mapsto (\omega+\gamma'(t),t)$ is a continuous map $\Omega \times[0,1]$ to itself. One then only needs to check that $\varphi$ is well-defined, which one can see from
\begin{align*}
\varphi([\omega,1])
& = [\omega+\gamma'(1),1] \\
& = [\omega+A^{-1}b,1] \\
& = [A\omega+b,0] \\
& = \varphi([A\omega+b,0]).
\end{align*}
Thus, $\varphi$ is well-defined and continuous. One can check that $\varphi$ is invertible with continuous inverse, concluding the argument. 

(d) This follows from part~(c) exactly as in \cite{DFGap}.
\end{proof}

\begin{lemma}
\label{Lemma2}
With the same assumptions and notation as in Lemma~\ref{Lemma1}, we have
\begin{equation}
\frF_{\nu}([g_{\chi,\beta}])= \beta \text{ for every } \chi \in K, \ \beta \in \pi^{-1}(\chi b).
\end{equation}
\end{lemma}

\begin{proof}
The proof is analogous to the proof of \cite[Lemma~8.3]{DFGap}. Fix $\chi$ and $\beta$, and denote $\phi = g_{\chi,\beta}$. For each $x = [\omega,s] \in X$, let us define ${\phi_x:\bbR\rightarrow \bbR}$ by $\phi_x(t)=\phi(\tau^t x)$. Its lift $\widetilde{\phi}_x:\bbR \rightarrow \bbR$ is a continuous function chosen so that $\phi_x=\pi \circ \widetilde{\phi}_x$ where $\pi$ is the canonical projection map. Using \eqref{eq:gknbDef}, we see that we may take 
\begin{equation} \label{eq:phibetalift}
\widetilde\phi_x(t) = \widetilde{\chi\omega} + \beta(s+t), \quad t \in \bbR,
\end{equation}
where $\widetilde{\chi\omega} \in \bbR$ is any element from $\pi^{-1}(\chi \omega)$.
For $\nu$-a.e.\ $x$, we have
\begin{align*}
\frF_{\nu}(\phi)
 = \lim_{t \to \infty} \frac{\widetilde{\phi}_x (t)}{t}.
 \end{align*}
  Fixing such an $x$ and using \eqref{eq:phibetalift} gives
  \begin{align*}
\frF_{\nu}(\phi)
  = \lim_{t \to \infty}\frac{ \widetilde{\chi\omega} + \beta(s+t)}{t} = \beta , \end{align*}
which finishes the proof.
\end{proof}

\begin{proof}[Proof of Theorem~\ref{t:main}]
This is a consequence of Lemmas~\ref{Lemma1} and \ref{Lemma2}.
\end{proof}

\section{The Doubling Map on the Dyadic Solenoid} \label{sec:solenoid}
 As a consequence of Theorem~\ref{t:main}, we can compute labels associated with affine autormorphisms on solenoids. For instance, we can show that the set of labels associated with the doubling map on the dyadic solenoid (see Subsection~\ref{subsec:solenoid}) is $\bbZ$:

\begin{corollary} \label{coro:2adicSchwartzGroup}
Let $\Omega$ be the dyadic solenoid, $\Omega=\bbR\times \bbZ_2 / \{(a,-a)\text{ : } a\in \bbZ\}$. Let $T$ be the doubling map on the solenoid, $T([r,s])=[2r,2s]$  and $\mu$ be a $T$-ergodic measure on $\Omega$, then
\begin{equation}
\frA(\Omega,T,\mu) = \bbZ.
\end{equation}
\end{corollary}

At first glance, one may find Corollary~\ref{coro:2adicSchwartzGroup} somewhat surprising. Indeed, the main result of \cite{DF22Doubling} was the conclusion that $\frF([\phi]) \in \bbZ$ for any $\phi \in C(X,\bbT)$ that factors through the (non-invertible) doubling map on the circle (recall that $\frF$ denotes the Schwartzman homomorphism). At the time, the authors suspected that that conclusion did not hold for more general functions on the solenoid and that one could leverage the local Cantor structure to produce continuous invariant sections having nonintegral rotation numbers. However, Corollary~\ref{coro:2adicSchwartzGroup} implies that this is not the case.

Theorem~\ref{t:main} applies to the doubling map on the dyadic solenoid.

\begin{proposition}\label{solenoid}
The solenoid, $\mathcal{S}_1$ is a compact connected abelian group, the doubling map $T_1$ is an automorphism on $\mathcal{S}_1$ and there exists a fully supported $T_1$-ergodic measure on $\mathcal{S}_1$.
\end{proposition}

\begin{proof}
It is known that the dyadic solenoid is a compact connected abelian group. It is then straightforward to check that the doubling map, $T_1$, is an automorphism.
 All that remains is to construct a fully-supported $T_1$-ergodic measure on $\mathcal{S}_1$. As discussed in \cite{DFZ} there exists a Bowen--Margulis measure (as well as a Sinai--Ruelle--Bowen measure), $\mu$, on $\mathcal{S}_3$ that is fully supported and $T_3$-ergodic.  Pushing this measure forward using the maps discussed in the proof of Proposition~\ref{prop:allSolenoidsEquiv}, we obtain a measure on $\mathcal{S}_1$ with the desired properties. \end{proof}

\begin{proof}[Proof of Corollary~\ref{coro:2adicSchwartzGroup}]
Proposition~\ref{solenoid} implies that we can apply Theorem~\ref{t:main} to construct the set of labels associated with $(\mathcal{S}_1,T_1,\mu_1)$ where $\mu_1$ is a fully supported $T_1$-ergodic measure on $\mathcal{S}_1$. The result follows by noting that $K$ defined in Theorem~\ref{t:main} is trivial in this case. Indeed, if $\chi \in K$, then $\chi(2\omega) = \chi(\omega)$ for all $\omega$, which then forces $\chi(\omega)=0$.
\end{proof}    

\section{Applications to Jacobi Matrices} \label{sec:Jacobi}

As a byproduct of our results, we are able to answer a question posed in \citep{DFZ} about gap labels for ergodic Jacobi matrices defined by the doubling map. Recall that a dynamically defined family of Jacobi matrices is specified by
\begin{equation}\label{e.jacmat}
[J_\omega\psi](n) = \overline{p(T^{n-1}\omega)}\psi(n-1)+q(T^n\omega)\psi(n) + p(T^n\omega)\psi(n+1),
\end{equation}
where $\omega \in \Omega$, a compact metric space, $T:\Omega \to \Omega$ is a homeomorphism, $p \in C(\Omega,\bbC)$, and $q \in C(\Omega,\bbR)$. Fixing a fully supported $T$-ergodic Borel probability measure on $\Omega$, the density of states measure and integrated density of states are given by \eqref{e.DOSMv2} and \eqref{e.IDSv2}.

\begin{corollary} \label{coro:JacobiLabels}
Suppose $\Omega$ is a compact connected abelian group, $A \in \Aut(\Omega)$, $b$ is in the path component of the identity of $\Omega$, and $\mu$ is a fully supported $T_{A,b}$-ergodic probability measure on $\Omega$. Given continuous $p$ and $q$, let $\{J_\omega\}$ denote the associated family of Jacobi operators as in \eqref{e.jacmat}. Then,
\[ k(E) \in \set{ \beta :  \beta \in \pi^{-1}(\chi b) \text{ for some } \chi \in \ker(\widehat{A}-I) } 
\]
for all $E \in \bbR \setminus \Sigma$, where $\Sigma$ denotes the almost-sure spectrum of the family $\{J_\omega\}_{\omega \in \Omega}$ and $k$ denotes the associated IDS.
\end{corollary}
\begin{proof}
This is a consequence of Theorem~\ref{t:main} and \cite[Theorem~1.1]{DFZ}.
\end{proof}

\begin{corollary} \label{coro:JacobiDoubling}
Suppose $(\Omega,T)$ denotes the doubling map on the dyadic solenoid and $\mu$ is a fully supported $T$-ergodic measure. Given continuous $p$ and $q$, let $\{J_\omega\}$ denote the associated family of Jacobi operators as in \eqref{e.jacmat}. Then, the almost-sure spectrum $\Sigma$ is connected.
\end{corollary}

\begin{proof}
This is a consequence of Corollaries~\ref{coro:2adicSchwartzGroup} and \ref{coro:JacobiLabels}.
\end{proof}

Using Corollary~\ref{coro:JacobiDoubling}, we can apply the method of \cite{DF22Doubling} to see that half-line Jacobi matrices dynamically defined by the doubling map on $\bbT$ have no gaps in their essential spectra, which extends \cite[Theorem~1.6]{DFZ} to the more general setting. As discussed in \cite[Remark~4.1]{DFZ}, this could not be done solely with the work in \cite{DFZ} and required new insights.  

Suppose $\Omega = \bbT$, $T:\Omega \to \Omega$ is the doubling map $T\omega = 2\omega$, and $\mu$ denotes Lebesgue measure on $\bbT$. Given continuous $p$ and $q$, let $\{J_\omega\}$ denote the associated family of half-line Jacobi operators given by 
\begin{equation}
[J_\omega  u](n) = \begin{cases}
\overline{p(T^{n-1}\omega)}u(n-1) + q(T^n\omega)u(n)  + p(T^n\omega)u(n+1) & n > 0 \\
q(\omega)u(0) + p(\omega)u(1) & n=0.
 \end{cases}
\end{equation}
Notice that we cannot discuss whole-line operators in this setting, since $T$ is not invertible.  Operators generated by the doubling map on $\bbT$ have been studied by a number of authors; see, for example, \cite{AviDamZha2020, Bjerkloev2020, BourgainSchlag2000CMP, ChulaevskySpencer1995, DF22Doubling, DFZ, DamKil2005b, Zhang2016}.

Given this setup, there exists $\Sigma \subseteq \bbR$ such that $\sigma_{\rm ess}(J_\omega) = \Sigma$ for $\mu$-a.e.\ $\omega \in \Omega$.   In \cite{DFZ}, the authors asked whether this almost-sure essential spectrum may have any gaps. 

\begin{corollary} \label{coro:JacobiDoubling2}
With $\Omega$, $T$, $p$, $q$, and $\Sigma$ as in the previous paragraph, $\Sigma$ is connected.
\end{corollary}

\begin{proof}
This follows from Corollary~\ref{coro:JacobiDoubling} and a repetition of the arguments in \cite[Section~4]{DFZ}.
\end{proof}

\begin{appendix}
\section{A Disconnected Example}
We start with the following helpful observation, which applies to arbitrary finite dynamical systems.
\begin{theorem} \label{t:frAforFinite}
Suppose $\Omega$ is a finite set with the discrete topology, $T:\Omega \to \Omega$ is a bijection, and $\mu$ is a $T$-ergodic probability measure on $\Omega$. One has
\begin{equation}
\frA(\Omega,T,\mu) = \frac{1}{p}\bbZ,
\end{equation}
where $p = \#\supp\mu$.
\end{theorem}
\begin{proof}
By ergodicity, $\supp\mu = \{\omega,T\omega,\ldots,T^{p-1}\omega\}$ for some $\omega \in \supp\mu$, so $(\supp\mu,T|_{\supp \mu})$ is conjugate to the shift $S:\omega \mapsto \omega+1$ on $\bbZ/p\bbZ$. Thus, using Proposition~\ref{prop:schwartzConj} and \cite[Proposition~6.1]{DFGap}, one has 
\[
\frA(\Omega,T,\mu) = \frA(\supp\mu,T|_{\supp \mu},\mu) = \frA(\bbZ/p\bbZ,S,\eta) = p^{-1}\bbZ,
\]
where $\eta$ denotes normalized counting measure on $\bbZ/p\bbZ$.
\end{proof}
\begin{corollary} \label{coro:disconExample}
There exists a disconnected group $\Omega$, an affine transformation $T=T_{A,b}:\Omega \to \Omega$ with $A \in \Aut(\Omega)$ and $b$ in the path component of the identity, and a $T$-ergodic measure $\mu$ on $\Omega$ such that
\begin{equation} \label{eq:broadSchwFails}
\frA(\Omega,T,\mu) \neq \set{ \beta :  \beta \in \pi^{-1}(\chi b) \text{ for some } \chi \in \ker(\widehat{A}-I) }.
\end{equation}
 In particular, the assumption of connectedness cannot be removed from Theorem~\ref{t:main}.
\end{corollary}
\begin{proof}
 Let us consider $\Omega = \bbZ/p\bbZ$ with the discrete topology. It is well-known and not hard to check that the affine homeomorphisms on $\Omega$ have the form $T\omega = A\omega+b$ with $A \in (\bbZ/p\bbZ)^\times$ (the group of units modulo $p$) and $b \in \Omega$. Since $\Omega$ is totally disconnected, $b$ is in the path component of $0$ if and only if $b=0$, so we only consider $T_A:\omega \mapsto A\omega$. The dual group of $\Omega$ is also isomorphic to $\Omega$ via the identification $m \in \Omega \leftrightarrow \chi_m \in \widehat{\Omega}$ where $$\chi_m(\omega)  =  m \omega/p.$$ With this identification, one can check that $\widehat{A}\chi_m = \chi_{Am}$ by the following direct calculation:
 \begin{equation*}
 [\widehat A \chi_m](\omega) = \chi_m(A\omega) = m(A\omega)/p = (Am)\omega/p = \chi_{Am}(\omega).
 \end{equation*}
  In particular, writing $K = \ker(\widehat{A}-I)$, we have
\[K = \set{\chi_m : Am=m \ \mathrm{mod} \ p\bbZ}.\]
For the automorphism $T_A:\omega  \mapsto A\omega$ on $\Omega = \bbZ/p\bbZ$ with $A$ a unit modulo $p$, let us denote the right-hand side of \eqref{eq:broadSchwFails} by $\mathfrak{G}(p,A)$. Choosing $\mu$ to be normalized counting measure on $\{1,2\}$, we see that $\frA(\bbZ/3\bbZ,T_2,\mu) = \frac{1}{2}\bbZ \neq \bbZ = \mathfrak{G}(3,2)$.
\end{proof}
\end{appendix}

\bibliographystyle{abbrv}
\bibliography{ref}
\end{document}